\documentclass[12pt,reqno]{article}
\usepackage{amsmath,amsthm,amsfonts,amssymb,amscd}
\usepackage{psfrag}
\usepackage{textcomp}
\usepackage[pctex32]{color}
\usepackage[pagewise]{lineno}
\usepackage{ifpdf}

\theoremstyle{plain}
\newtheorem{thm}{Theorem}[section]

\newtheorem{prop}[thm]{Proposition}

\newtheorem{defi}[thm]{Definition}

\newtheorem{maintheorem}{Theorem}

\newcommand{\re}{{\mathbb{R}}}
\newcommand{\nat}{{\mathbb{N}}}

\title{Intransitive sectional-Anosov flows on 3-manifolds}
\author{S. Bautista, A. M. L\'opez B., H. M. S\'anchez 
        \thanks{
This work is partially supported by CAPES, Brazil.}}
\date{}

\begin{document}
\maketitle

\begin{abstract}
For each $n\in \mathbb{Z}^+$, we show the existence of Venice masks \index{Venice mask}
(i.e. intransitive sectional-Anosov flows with dense periodic orbits,
\cite{bmp}, \cite{mp2}, \cite{lec}, \cite{ls}) containing $n$ equilibria on certain
compact $3$-manifolds. These examples are characterized because of the maximal invariant set is a finite union
of homoclinic classes. Here, the intersection between two different homoclinic classes is 
contained in the closure of the union of unstable manifolds of the singularities.
\end{abstract}


\section{Introduction}

The dynamical systems theory proposes to find 
qualitative information on the behavior in a determinated system without actually finding the solutions.
An example of this study is the phenomenom of transverse homoclinic points. Birkhoff proved that any
transverse homoclinic orbit is accumulated by periodic points.
With the introduction of {\em uniformly hyperbolic dynamical systems} by Smale \cite{sma}, was 
developed a study of robust models containing infinitely many periodic orbits.
However, the uniform hyperbolicity were soon realized to be less universal properties that 
was initially thought. On the other hand, there exist many classes of systems non hyperbolic that are often the
specific models coming from concrete applications. This motivated weaker formulations of 
hyperbolicity such as existence of dominated splitting, partial hyperbolicity or sectional hyperbolicity.

Recall that a $C^r$ vector field $X$ in $M$ is $C^r$ robustly transitive or $C^r$ robustly periodic depending
on whether every $C^r$ vector field $C^r$ close to it is transitive or has dense periodic
orbits. Are particularly interesting the sectional-hyperbolic sets and 
sectional-Anosov flows, which were introduced in \cite{mo} and \cite{mem} respectively 
as a generalization of the hyperbolic sets and Anosov flows. Their importance is because of the robustly transitive
property in dimension three of attractors sectional-Anosov flows \cite{mpp4}, and the inclusion of important examples
such as the saddle-type hyperbolic attracting sets, the singular horseshoe, the geometric and multidimensional 
Lorenz attractors \cite{abs}, \cite{bpv}, \cite{gw}.

With respect to robustly transitive property, we can mention a clue fact in the scenario of sectional-Anosov flows. 
As consequence of the main result in \cite{apu}, and {\em Theorem 32} in \cite{lec} follows that every sectional-Anosov 
flow with a unique singularity on a compact 3-manifold is $C^r$ robustly
periodic if and only if is $C^r$ robustly transitive. On the other hand, there is not equivalence between transitivity and 
density of periodic orbits (which is true for Anosov flows). Indeed, there
exist examples of sectional-Anosov flows non transitive with dense periodic orbits supported on compact three dimensional 
manifolds. So, a sectional-Anosov flow is said a {\em Venice mask} if it has dense periodic orbits but is not transitive. 
An example of Venice mask with a unique singularity was given in \cite{bmp}, and for three singularities was
provided in \cite{mp}. Recently, \cite{ls} showed the construction the examples with two equilibria. They 
are characterized because the maximal invariant set is non 
disjoint union of two homoclinic classes \cite{mp}, \cite{mp2}, \cite{bmp}, and the intersection between these classes is 
contained in the closure of the union of unstable manifolds of the singularities.

Particularly, was proved in \cite{mp}, \cite{mp2} that every Venice mask with a unique singularity 
has these properties.
The above observations motivate the following questions,

\begin{enumerate}
	\item It is possible to obtain Venice masks with more singularities?
	\item The maximal invariant set of every Venice mask is union of two homoclinic classes?
	\item How is the intersection of these homoclinic classes?
\end{enumerate}

The answer to the first question is positive. We use the ideas developed in \cite{ls} and \cite{mp} for the construction
of these examples, which provide more tools and clues for a general theory of Venice masks. In particular, we construct 
an example with five singularities which is non disjoint union of three homoclinic classes. So, the answer to the second
question is false.

Let us state our results in a more precise way.

Consider a Riemannian compact manifold $M$ of dimension three (a {\em compact $3$-manifold} for short).
$M$ is endowed with a Riemannian metric $\langle\cdot,\cdot\rangle$ and an
induced norm $\lVert\cdot\rVert$. 
		We denote by $\partial M$ the boundary of $M$.
		Let ${\cal X}^1(M)$ be the space
		of $C^1$ vector fields in $M$ endowed with the
		$C^1$ topology.
		Fix $X\in {\cal X}^1(M)$, inwardly
		transverse to the boundary $\partial M$ and denotes
		by $X_t$ the flow of $X$, $t\in I\!\! R$.\\
		
The {\em omega-limit set} of $p\in M$ is the set
$\omega_X(p)$ formed by those $q\in M$ such that $q=\lim_{n \rightarrow \infty}X_{t_n}(p)$ for some
sequence $t_n\to\infty$. The {\em alpha-limit set} of $p\in M$ is the set
$\alpha_X(p)$ formed by those $q\in M$ such that $q=\lim_{n \rightarrow \infty}X_{t_n}(p)$ for some
sequence $t_n\to -\infty$. Given $\Lambda \in M$ compact, we say that $\Lambda$ is {\em invariant} 
if $X_t(\Lambda)=\Lambda$ for all $t\in I\!\! R$.
We also say that $\Lambda$ is {\em transitive} if
$\Lambda=\omega_X(p)$ for some $p\in \Lambda$; {\em singular} if it
contains a singularity and
{\em attracting}
if $\Lambda=\cap_{t>0}X_t(U)$
for some compact neighborhood $U$ of it.
This neighborhood is often called
{\em isolating block}.
It is well known that the isolating block $U$ can be chosen to be
positively invariant, i.e., $X_t(U)\subset U$ for all
$t>0$.
An {\em attractor} is a transitive attracting set.
An attractor is {\em nontrivial} if it is
not a closed orbit.\\

The {\em maximal invariant} set of $X$ is defined by
	$M(X)= \bigcap_{t \geq 0} X_t(M)$.

\begin{defi}\label{defhyp}
		\label{hyperbolic}
		A compact invariant set $\Lambda$ of $X$ is {\em hyperbolic}
		if there are a continuous tangent bundle invariant decomposition
		$T_{\Lambda}M=E^s\oplus E^X\oplus E^u$ and positive constants
		$C,\lambda$ such that

		\begin{itemize}
		\item $E^X$ is the vector field's
		direction over $\Lambda$.
		\item $E^s$ is {\em contracting}, i.e.,
		$
		\lVert DX_t(x) \left|_{E^s_x}\right.\rVert
		\leq Ce^{-\lambda t}$,
		for all $x \in \Lambda$ and $t>0$.
		\item $E^u$ is {\em expanding}, i.e.,
		$
		\lVert DX_{-t}(x) \left|_{E^u_x}\right.\rVert
		\leq Ce^{-\lambda t},
		$
		for all $x\in \Lambda$ and $t> 0$.
		\end{itemize}
\end{defi}
A compact invariant set $\Lambda$ has a {\em dominated splitting} with
respect to the tangent flow if there are an invariant splitting 
$T_{\Lambda}M = E\oplus F$ and positive numbers $K,\lambda$
such that

$$\lVert DX_t(x)e_x\rVert\cdot \lVert f_x\rVert\leq Ke^{-\lambda t} \lVert DX_t(x)f_x\rVert\cdot \lVert e_x\rVert,\qquad
\forall x\in\Lambda, t \geq 0, (e_x , f_x) \in E_x\times F_x .$$

Notice that this definition allows every compact invariant set $\Lambda$
to have a dominated splitting with respect to the tangent flow (See \cite{bamo}):
Just take $E_x = T_xM$ and $F_x = 0$,
for every $x\in\Lambda$ (or $E_x = 0$ and $F_x = T_x M$ for 
every $x\in\Lambda$). 

A compact invariant set $\Lambda$ is {\em partially hyperbolic} if
it has a {\em partially hyperbolic splitting}, i.e., a dominated splitting $T_{\Lambda}M = 
E\oplus F$ with respect to the tangent flow
whose dominated subbundle $E$ is contracting in the sense of
Definition \ref{defhyp}.\\

The Riemannian metric $\langle\cdot ,\cdot\rangle$ of $M$
induces a $2$-Riemannian metric \cite{mv},
$$\langle u, v/w\rangle_p= \langle u, v\rangle_p\cdot
\langle w, w\rangle_p - \langle u, w\rangle_p\cdot \langle v, w\rangle_p,\quad
\forall p\in M, \forall u, v, w \in T_p M.$$
This in turns induces a 2-norm \cite{gah} (or areal metric \cite{kat}) defined by

$$\lVert u, v\rVert =\sqrt{\langle u, u/v\rangle_p}
\qquad \forall p\in M, \forall u, v \in T_p M.$$ 

Geometrically, $\lVert u, v\rVert$ represents the area of the parallelogram 
generated by $u$ and $v$ in $T_p M$.\\

If a compact invariant set $\Lambda$ has a dominated splitting $T_{\Lambda}M = F^s \oplus F^c$ with
respect to the tangent flow, then we say that its central subbundle $F^c$ is {\em sectionally
expanding} if

$$\lVert DX_t (x)u, DX_t (x)v\rVert \geq K^{-1}e^{\lambda t}
\lVert u, v\rVert, \quad\forall x \in \Lambda, u,v\in F^c_x, t \geq 0.$$

Recall that a singularity of a vector field is hyperbolic if the eigenvalues of its linear part
have non zero real part. \\

By a {\em sectional hyperbolic splitting} for $X$ over $\Lambda$ we mean a partially hyperbolic
splitting $T_{\Lambda}M = F^s\oplus F^c$ whose central subbundle $F^c$ is sectionally expanding.

\begin{defi}
A compact invariant set $\Lambda$ is {\em sectional hyperbolic} for $X$ if 
its singularities are hyperbolic and if there is a sectional hyperbolic splitting for $X$ over $\Lambda$.
\end{defi}

\begin{defi}
\label{secflow}
We say that $X$ is a {\em sectional-Anosov flow} if $M(X)$ is a sectional hyperbolic
set.
\end{defi}

The Invariant Manifold Theorem \cite{hps} asserts that if $x$ belongs
to a hyperbolic set $H$ of $X$, then the sets

$$W^{ss}_X(p)  =  \{x\in M:d(X_t(x),X_t(p))\to 0, t\to \infty\} \qquad and$$
$$W^{uu}_X(p)  =  \{x\in M:d(X_t(x),X_t(p))\to 0, t\to -\infty\},$$

are $C^1$ immersed submanifolds of $M$ which are tangent at $p$ to the subspaces $E^s_p$ and $E^u_p$ of $T_pM$ respectively.

$$W^{s}_X(p) =   \bigcup_{t\in \re}W^{ss}_X(X_t(p))\qquad and\qquad W^{u}_X(p)  =   \bigcup_{t\in \re}W^{uu}_X(X_t(p))$$

are also $C^1$ immersed submanifolds tangent to $E^s_p\oplus E^X_p$ and $E^X_p\oplus E^u_p$ at $p$ respectively.\\

We denote by $Sing(X)$ to the set of singularities of $X$,  and $Cl(A)$ to the closure of $A$.

\begin{defi}
\label{ll}
We say that a singularity $\sigma$ of a sectional-Anosov flow $X$ of dimension three is {\em Lorenz-like}
if  it has three real eigenvalues $\lambda^{ss},\lambda^{s},\lambda^u$ with $\lambda^{ss}<\lambda^s<0<-\lambda^s<\lambda^u$.
The strong stable foliation  associated to $\sigma$ and denoted by $\mathcal{F}^{ss}_X(\sigma)$, is the foliation contained 
in $W^s(\sigma)$ which is tangent
to space generated by the eigenvalue $\lambda^{ss}$.
\end{defi}

We denote as $W^s(Sing(X))$ to $\bigcup\limits_{\sigma\in Sing(X)}W^s(\sigma)$.

Respectively, $W^u(Sing(X))=\bigcup\limits_{\sigma\in Sing(X)}W^u(\sigma)$.

\begin{defi}
A periodic orbit of $X$ is the orbit of some $p$ for which there is a minimal
$t > 0$ (called the period) such that $X_t(p) = p$.
\end{defi}

$\gamma$ is a {\em transverse homoclinic orbit} of a hyperbolic periodic orbit $O$ if $\gamma\subset W^s(O)\cap W^u(O)$,
and $T_qM = T_qW^s(O) + T_qW^u(O)$ for some (and hence all) point $q\in\gamma$. 
The homoclinic class $H(O)$ of a hyperbolic periodic orbit $O$ is the closure of the 
union of the transverse homoclinic orbits of $O$. 
We say that a set $\Lambda$ is a homoclinic class if $\Lambda = H(O)$ for some hyperbolic periodic orbit $O$.

\begin{defi}
A Venice mask is a sectional-Anosov
flow with dense periodic
orbits which is not transitive.
\end{defi}

$D^n$ denotes the unit ball in $\re^n$ and $\partial D^n$ the boundary of $D^n$.

An $n$-cell\index{n-cell} is a manifold homeomorphic to the open ball $D^n\setminus\partial D^n$.\\

The following definition appears in \cite{hem}.

\begin{defi}
A handlebody of genus $n\in \nat$ (or a cube with $n$-handles) is a compact 3-manifold with boundary $HB_n$ such that

\begin{itemize}
 \item $HB_n$ contains a disjoint collection of $n$ properly embedded 2-cells.
 \item A 3-cell is obtain of cutting $HB_n$ along the boundary of these 2-cells.  
\end{itemize}

\end{defi}

Observe that a 3-ball is a handlebody of genus 0, whereas a solid torus is a handlebody of genus 1.\\ 

In \cite{mor3} was proved that every orientable handlebody $HB_n$ of genus $n\geq 2$ supports a transitive 
sectional-Anosov flow. In particular these flows have $n-1$ singularities. An example is the geometric 
Lorenz attractor which is supported on a solid bitorus. We show that certain classes of handlebody support 
Venice masks.\\

The following is the statement of the main result in this paper. 

\begin{maintheorem}
\label{thF'}
For each $n\in\mathbb{Z}^+$: 
\begin{itemize}
\item There exists a Venice mask $X_{(n)}$ with $n$ singularities which is supported on some compact 3-manifold $M$.  
In addition, $M(X_{(n)})$ can be decomposed as union of two homoclinic classes.
\item There exists a Venice mask $Y_{(n)}$ supported on some compact 3-manifold $N$ such that $N(Y_{(n)})$ is union of
$n+1$ homoclinic classes.

In both cases, the intersection of two different homoclinic classes of the maximal invariant set, 
is contained in the closure of the union of unstable manifolds of the singularities.
\end{itemize}
\end{maintheorem}

In section \ref{prel}, we briefly describe the construction and some important properties
for the known examples with two and three singularities. In section \ref{expar}, by using 
the techniques of the Venice masks with two singularities, we construct an example with four singular points.
In the same way, in Section \ref{eximpar}, from the Venice mask with three singularities, will be obtained an example with 
five equilibria being its maximal invariant set union of three homoclinic classes. Theorems \ref{thimpar} and \ref{thpar} 
will be obtained of a inductive process. Finally, {\em Theorem} \ref{thF'} will be a direct consequence of 
{\em Theorem} \ref{thimpar} and {\em Theorem} \ref{thpar}.



\section{Preliminaries}
\label{prel}

We make a brief description about the known Venice masks.\\

An example with a unique singularity was given in \cite{bmp}, and in \cite{mp} was proved that every Venice mask $X_{(1)}$ 
with one equilibrium satisfies the following properties:
\begin{itemize} 
 \item $M(X_{(1)})$ is union of two homoclinic classes $H^1_{X_{(1)}}$, $H^2_{X_{(1)}}$.
 \item $H^1_{X_{(1)}}\cap H^2_{X_{(1)}}=Cl(W^u_{X_{(1)}}(\sigma))$ where $\sigma$ is the singularity of $X_{(1)}$.
\end{itemize}

In \cite{ls} were exhibited two Venice masks containing two equilibria $\sigma_1$, $\sigma_2$.

For the first example we have a vector field $X$ verifying:
\begin{itemize}
 \item $M(X)$ is the union of two homoclinic class $H_X^1$, $H_X^2$.
 \item $H_X^1\cap H_X^2=O$, where $O$ is a hyperbolic periodic orbit.
 \item $O=\omega_X(q)$, for all $q\in W^u_X(\sigma_1)\cup W^u_X(\sigma_2)\setminus\{\sigma_1,\sigma_2\}$.
\end{itemize}

The vector field \index{Vector field}
$Y$ that determines the second example with two singularities $\sigma_1$, $\sigma_2$ satisfies:
\begin{itemize}
 \item $M(Y)$ is the union of two homoclinic class $H_Y^1$, $H_Y^2$.
 \item $H_Y^1\cap H_Y^2=Cl(W^u_Y(\sigma_1)\cup W^u_Y(\sigma_2))$.
\end{itemize}

An essential element to obtain the examples with two singularities is the existence of a return map defined in
a cross section \index{Cross section}
$R$. A foliation $\mathcal{F}$ is defined on $R$, which has vertical segments in the rectangular components 
$B, C, D, E$ and radial segments in the annuli components $A,F$. \\

We are interested to take a $C^{\infty}$ two-dimensional map $\tilde{G}:R\setminus\{d^-, d^+\}\to Int(R)$ such as
in \cite{ls}, satisfying
the hypotheses \textbf{(L1)-(L3)} established there. In particular, \textbf{(L1)} and \textbf{(L2)} imply the contraction 
and the invariance of the leaf $l$ by $\tilde{G}$. So, the map $\tilde{G}$ has a fixed point $P\in l$.
We define $H^+=A\cup B\cup C$ and $H^-=D\cup E\cup F$. For

$$A_{\tilde{G}}^-=Cl\left(\bigcap_{n\geq 1} \tilde{G}^n(H^{-}) \right), \qquad 
A_{\tilde{G}}^+=Cl\left(\bigcap_{n\geq 1} \tilde{G}^n(H^{+}) \right)$$

follow that $A_{\tilde{G}}^+$ and $A_{\tilde{G}}^-$ are homoclinic classes and $\{P\}= A_{\tilde{G}}^+\cap A_{\tilde{G}}^-$.

\begin{figure}
\begin{center}
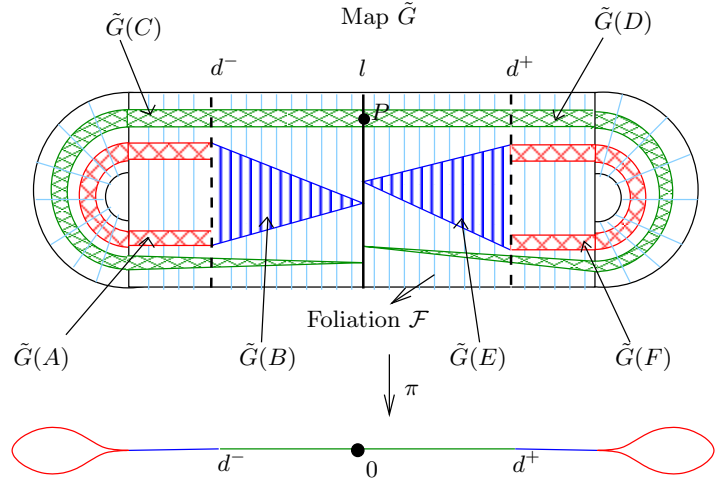
\caption{\label{7} Two-dimensional map $\tilde{G}$ on region $R$.}
\end{center}
\end{figure}
\pagebreak
\newpage

The mode to obtain the example with three singularities described in \cite{mp} is easier. First of all, is
important to know some properties about the dynamic of the Geometric Lorenz Attractor \index{Geometric Lorenz Attractor}
(GLA for short) \cite{gw}.\\

In \cite{b} was proved that this attractor is a homoclinic class. The result is obtained due to the existence of a 
return map $F$ for the flow, defined on a cross section $\Sigma$. This map preserves the stable foliation $\mathcal{F}^s$,
where the leaves are vertical lines. 
The induced map $f$ in the leaf space is differentiable and expansive.\\

\begin{figure}
\begin{center}
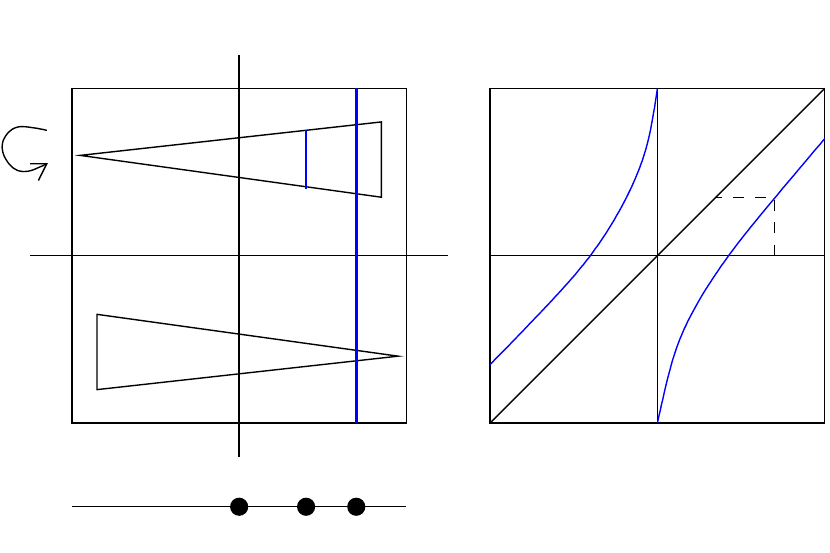
\caption{\label{ss} Map $F$.}
\end{center}
\end{figure}

\begin{figure}
\begin{center}
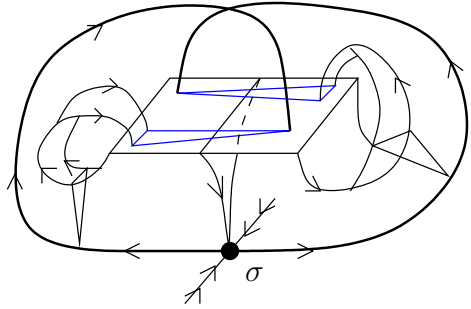
\caption{\label{1o} Geometric Lorenz Attractor.}
\end{center}
\end{figure}

The GLA is modified in \cite{mp} by adding two singularities to the flow located at $W^u(\sigma)$. We called 
this modification as $GLA_{mod}$. We glue together in a $C^{\infty}$ fashion two copies of this flow along the
unstable manifold of the singularity $\sigma$, thus generating the flow depicted in Figure
\ref{ex3}. In this way is obtained a sectional-Anosov flow $X_{(3)}$ with dense periodic
orbits and three equilibria whose maximal invariant set is non-disjoint union of two homoclinic classes. In this case, 
the intersection between the homoclinic classes is $Cl(W^u_{X_{(3)}}(\sigma))$.\\

Observe that this flow is supported on a handlebody \index{Handlebody} of genus 4.


\section{Decomposing the maximal invariant set as union of two homoclinic classes}
\label{expar}

\subsection{Vector field $Z$}

We provide an example with four singularities. We start with the vector field $X$ 
associated to the Venice mask with two singularities. Then, we construct a plug \index{Plug} $Z_4$
containing two additional equilibria $\sigma_3$, $\sigma_4$. In this way, the flow is obtained 
through plug $Z_4$ surgery from one solid tritorus onto another manifold exporting some of its properties.\\

The vector field $X$ is supported on a solid tritorus $ST_1$. Now, we remove a connected component $B$ of $ST_1$ as
in Figure \ref{plug1}.\\

\begin{figure}
\begin{center}
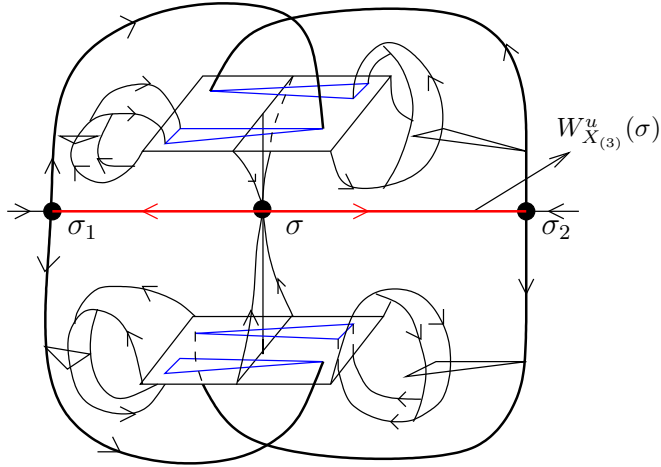
\caption{\label{ex3} Example with three singularities.}
\end{center}
\end{figure}

The behavior across the faces removed is similar with respect to observed in the example given by the vector field 
$Y$ in \cite{ls}. 

\pagebreak
\newpage

On face $1$, we identify three regions determinated by the singular leaves saturated by 
the flow. In the middle region on face $1$, the trajectories crossing inward to $\partial A$, such as the branch 
unstable manifold of the two initial singularities. All trajectories are crossing inward to face
$2$ as $\partial A$.

\begin{figure}
\begin{center}
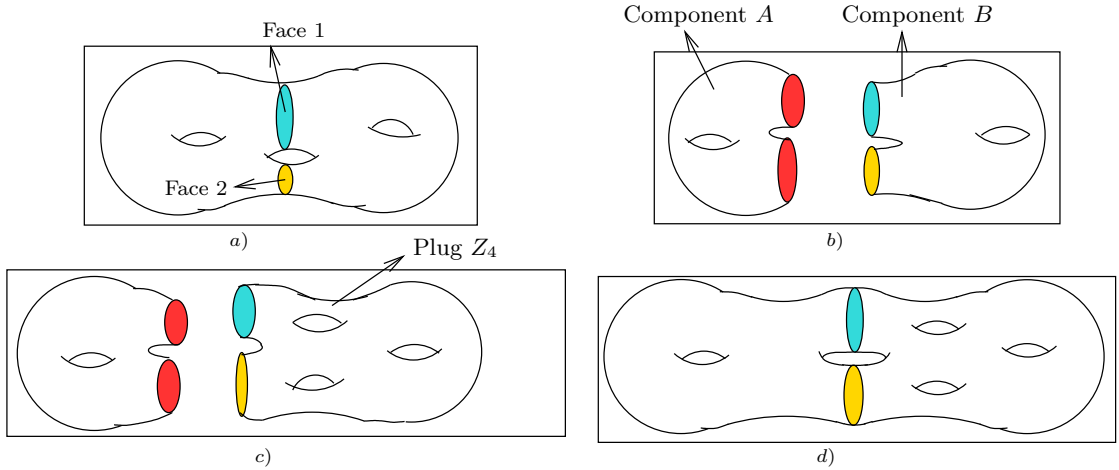	
\caption{\label{plug1} Steps by gluing the new plug.}
\end{center}
\end{figure}


\begin{figure}
\begin{center}
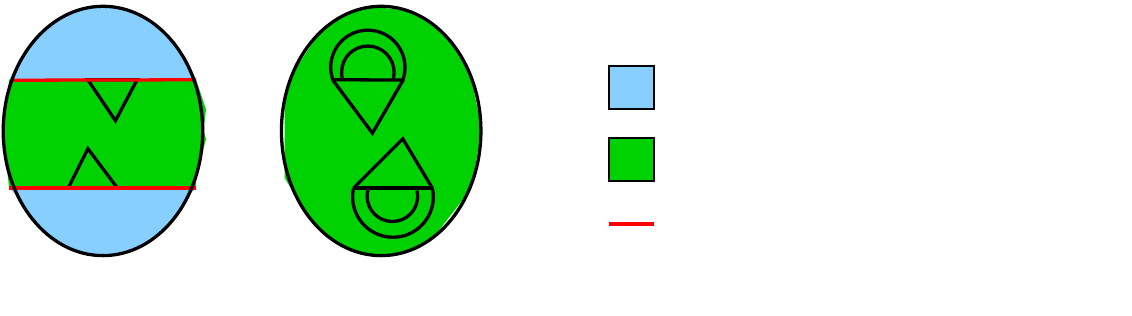
\caption{\label{faces2} Faces.}
\end{center}
\end{figure}


As we before mention, will be constructed an adequate plug $Z_4$ to include the additional equilibria. We ask the 
singularities to be Lorenz-like. In this sense, is considered the Cherry flow and its deformation via DA-Attractor
\cite{r} to obtain a perturbed Cherry flow such as in \cite{ls}. Take the neighborhood $U$ of $\sigma$ such as in 
\cite{ls}. So, $U$ contains a source $\sigma$ and two
saddles $\sigma_1$, $\sigma_2$. A disk $D$ centered in $\sigma$, with the flow outwardly transverse to the $\partial D$
is removed. After that, the vector field over $U$ in the perturbed
Cherry flow is multiplied 
by a strong contraction $\lambda_{ss}$. Then, two Lorenz-like singularities $\sigma_1$ and $\sigma_2$ are obtained and one
hole $H_{\sigma}$ is generated.
\index{Lorenz-like singularity}

\pagebreak

To continue, is made a saddle-type connection between the
branches of the unstable manifold of $\sigma_1$ and the stable manifold of a singularity $\sigma_3$. 
In a similar way is made a saddle-type
connection between the branches of the unstable manifold of $\sigma_2$ and the stable manifold of a singularity $\sigma_4$.  
On the other hand, such as in the GLA (see \cite{b}), two holes $H_{\sigma_1}$, $H_{\sigma_2}$ are generated by
the unstable manifolds of the singularities
$\sigma_1, \sigma_2$ respectively. In each hole, there is a singular point $\rho_i$ with two complex eigenvalues $z_i$, 
$\bar{z_i}$ ($Re(z_i)>0$) and an eigenvalue $\lambda_i<0$. 

\begin{figure}
\begin{center}
{\scalebox{1.00}{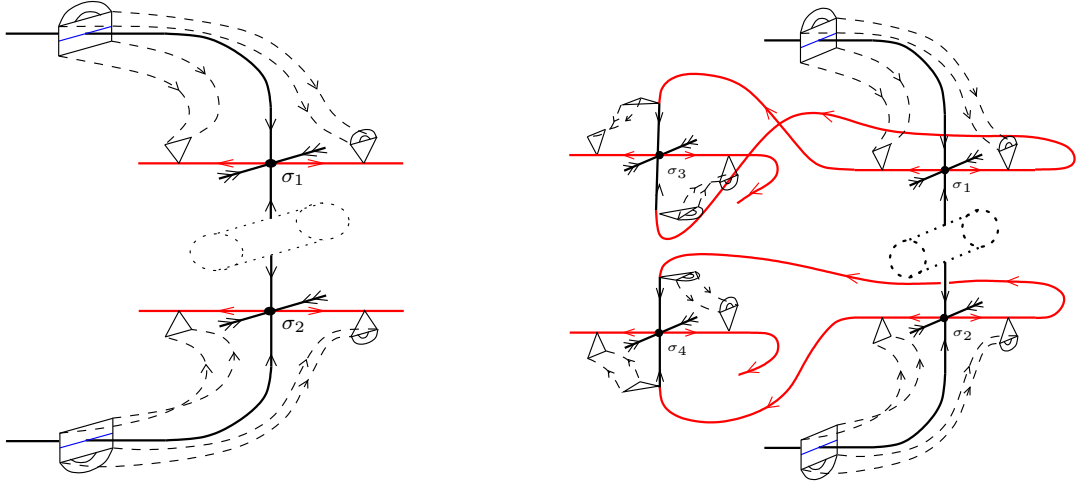}}
\caption{\label{plugZ_4} Plug $Z_4$.}
\end{center}
\end{figure}

The holes $H_{\sigma_1}$, $H_{\sigma_2}$
are connected with the third hole $H_{\sigma}$ separating $\sigma_1$ and $\sigma_2$. The branches of the unstable 
manifolds associated to $\sigma_3$ and $\sigma_4$ cross the faces 1 and 2, and these go to the Plug 3 defined in \cite{ls}
such as the example given by the vector field $X$. See Figure \ref{exa4}.

Therefore, we obtain a handlebody $HB_5$ of genus five. So, the vector field $Z$ produced by gluing plug $Z_4$ instead
the removed connected component $B$, satisfies $Z_t(HB_5)\subset Int(HB_5)$ for all $t > 0$. Moreover $Z$ is
transverse to the boundary handlebody.\\ 
 
\begin{figure}
\begin{center}
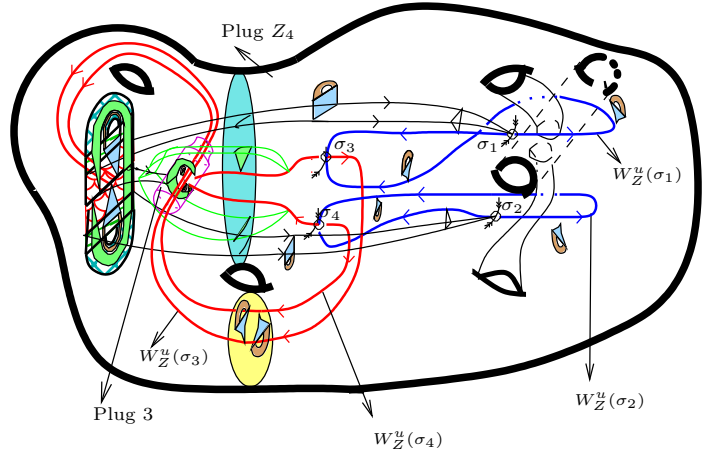
\caption{\label{exa4} Venice mask with four singularities.}
\end{center}
\end{figure}


We exhibit with details the behavior near to the singularities. For that, we mention some facts that appear in
\cite{exm}. As every singularity is Lorenz-like, there exists
a center unstable manifold \index{Center unstable manifold} $W^{cu}_Z(\sigma_i)$ associated to $\sigma_i$ $(i=1,2,3,4)$.
It is divided by $W^u_Z(\sigma_i)$ and $W^s_Z(\sigma_i)\cap W^{cu}_Z(\sigma_i)$ in the four sectors $s_{11}, s_{12}, 
s_{21}, s_{22}$. There is also a projection $\pi: V_{\sigma_i}\to W^{cu}_Z(\sigma_i)$ defined in a neighborhood 
$V_{\sigma_i}$ of $\sigma_i$ via the strong stable foliation of the maximal invariant set associated to flow.


For $\sigma\in Sing(Z)$, we define the matrix

\begin{align*}
A(\sigma)=
\left(
\begin{array}{cc}
a_{11} & a_{12} \\
a_{21} & a_{22}
\end{array}
\right),
\end{align*}

where

$$a_{ij}=
\begin{cases}
1\qquad \textit{if }\sigma\in Cl(\pi(M(Z)\cap V_{\sigma}))\cap s_{ij} \\
0\qquad \textit{if }\sigma\notin Cl(\pi(M(Z)\cap V_{\sigma}))\cap s_{ij}. 
\end{cases}
$$

$A(\sigma)$ does not depend on the chosen center unstable manifold $W^{cu}_Z(\sigma)$.

Figure \ref{cusigma} shows the case for the singularity $\sigma_1$ of the example.\\

These are the associated matrices to the singularities of our vector field $Z$.
$$A_{\sigma_{1}}=
\left(
\begin{array}{cc}
1 & 1 \\
0 & 0
\end{array}
\right),
\quad
A_{\sigma_{2}}=
\left(
\begin{array}{cc}
0 & 0 \\
1 & 1
\end{array}
\right),
\quad
A_{\sigma_{3}}=
\left(
\begin{array}{cc}
1 & 0 \\
0 & 1
\end{array}
\right),
\quad
A_{\sigma_4}=
\left(
\begin{array}{cc}
0 & 1 \\
1 & 0
\end{array}
\right).
$$

\begin{figure}
\begin{center}
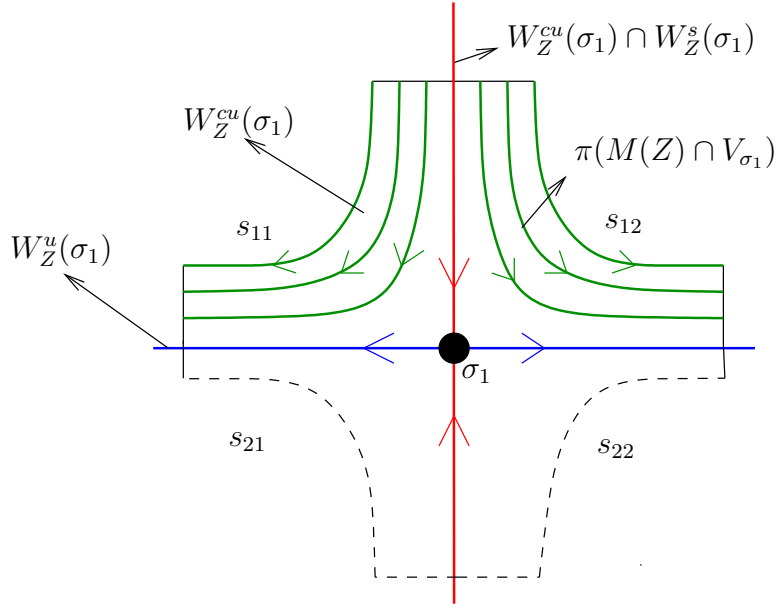
\caption{\label{cusigma} Center unstable manifold of $\sigma_1$.}
\end{center}
\end{figure}


Now, we consider the following hypotheses.

\textbf{($Z$1)}: There are two repelling periodic orbits $O_1$, $O_2$ in $Int(HB_5)$ crossing the
holes of $R$.

\textbf{($Z$2)}: There are two solid tori neighborhoods $V_1, V_2 \subset Int(HB_5)$ of $O_1,O_2$ with
boundaries transverse to $Z_{t}$ such that if $N_4 = HB_5 \setminus (V_1 \cup V_2)$, then $N_4$ is a
compact neighborhood with smooth boundary transverse to $Z_{t}$ and $Z_{t}(N_4)\subset
N_4$ for $t>0$. $N_4$ is a handlebody of genus five with two solid tori removed. 

\textbf{($Z$3)}: $R \subset N_4$ and the return map $\tilde{G}$ induced by $Z$ in $R$ satisfies the properties
\textbf{(L1)-(L3)} given in \cite{ls}. Moreover,

$$\{q\in N : Z_{t}(q) \notin R, \forall t\in \re\} = Cl(W^u_Z(\sigma_1)\cup W^u_Z(\sigma_2)).$$

\begin{figure}
\begin{center}
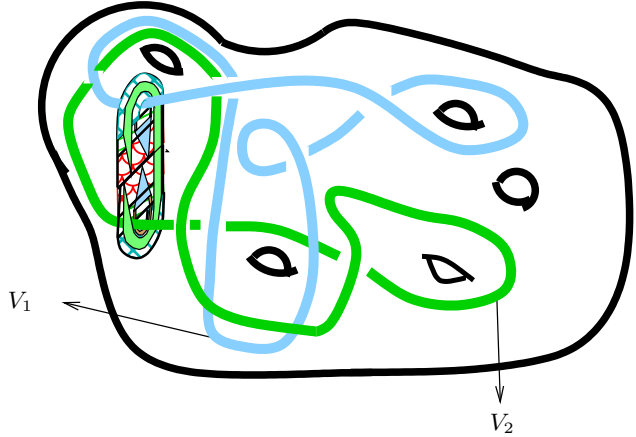
\caption{\label{exa4b} The manifold $N_4$.}
\end{center}
\end{figure}

We define

$$ A_Z^+=Cl\left(\bigcup_{t\in\re} Z_{t}(A^+_{\tilde{G}})\right)\qquad
\text{and}\qquad A_Z^-=Cl\left(\bigcup_{t\in\re} Z_{t}(A^-_{\tilde{G}})\right).$$

\begin{prop}
$Z$ is a Venice mask with four singularities supported on the compact 3-manifold $N_4$. $N_4(Z)$ is the union
of two homoclinic classes $A_Z^+$, $A_Z^-$. The intersection between $A_Z^+$ and $A_Z^-$ is a hyperbolic 
periodic orbit $O$ contained in $Cl(W^u_Z(\sigma_3)\cup W^u_Z(\sigma_4))$.
\end{prop}

\begin{proof}
By construction $Z$ has four singularities. The proof to be $A_Z^+$, $A_Z^-$ homoclinic classes is the same given 
in \cite{ls}. Also the fact to be $Z$ a Venice mask. The intersection between the homoclinic classes is reduced 
to a hyperbolic periodic orbit $O$ because of $\{P\} = A^+_{\tilde{G}}\cap A^-_{\tilde{G}}$ and by hypotheses \textbf{($Z$3)}. 
Here, $O=O_Z(P)$. We observe that the branches of
the unstable manifolds of $\sigma_3$ and $\sigma_4$ intersect the leaf $l$ of the foliation $\mathcal{F}$ in $R$.
Then the hypotheses \textbf{(L1)}, \textbf{(L2)} of the map $\tilde{G}$, and the invariance of the flow imply
$O\subset \omega_Z(q)$ for all regular point $q\in W^u_Z(\sigma_3)\cup W^u_Z(\sigma_4)$. As $W^u_Z(\sigma_3)\subset A_Z^+$
and $W^u_Z(\sigma_4)\subset A_Z^-$ (see {\em Proposition 4.1} \cite{ls}) we conclude
$A_Z^+\cap A_Z^-\subset Cl(W^u_Z(\sigma_3)\cup W^u_Z(\sigma_4))$.
\end{proof}

\subsection{General case}

We expose a general result. More specifically the following theorem holds. 

\begin{thm}
\label{thpar}
For every $n\in\mathbb{Z}^+$, there exists a Venice mask $X_{(n)}$ with $n$ singularities supported on a handlebody \index{Handlebody}
$N_n$ of genus $n+1$ with two solid tori removed. $N_n(X_{(n)})$ is the non-disjoint union of two homoclinic classes,
\index{Homoclinic classes} and the intersection between them is a hyperbolic periodic orbit \index{Hyperbolic periodic
orbit} contained in $Cl(W^u(Sing(X_{(n)})))$.
\end{thm}

\begin{proof}
$n=1,2,4$ is done. Consider $n\geq 3$. Again, we remove the same connected component $B$ to the manifold that 
supports the Venice mask $X$ with two equilibria.
We glue a plug $Z_n$ containing $n$ Lorenz-like singularities. For each singularity in $Z_n$, we have a
saddle-type  connection between $W^u_{X_{(n)}}(\sigma_i)$ and $W^s_{X_{(n)}}(\sigma_{i+2})$, $i=1,\ldots,n-2$.
For each saddle-type connection is produced a hole. Figure 
\ref{Plug $X_5$} exhibits the particular case for Plug \index{Plug} $Z_5$.
The branches of the unstable manifolds associated
to $\sigma_{n-1}$ and $\sigma_n$ cross the faces 1 and 2, and these go to the Plug 3.\\

So, the new manifold is a handlebody $HB_{n+1}$ of genus $n+1$ and supports a flow $X_{(n)_t}$ with $n$ equilibria. 
The flow is obtained by gluing plug $Z_n$ instead the connected component $B$. In this way, the vector field $X_{(n)}$
on $HB_{n+1}$ satisfies $X_{(n)_t}(HB_{n+1})\subset Int(HB_{n+1})$ for all $t > 0$. In addition, $X_{(n)}$ is
transverse to the boundary handlebody. 

Here, 

$$
A_{\sigma_{1}}=
\left(
\begin{array}{cc}
1 & 1 \\
0 & 0
\end{array}
\right),
\qquad 
A_{\sigma_{2}}=
\left(
\begin{array}{cc}
0 & 0 \\
1 & 1
\end{array}
\right),
\qquad
A_{\sigma_{2k-1}}=
\left(
\begin{array}{cc}
1 & 0 \\
0 & 1
\end{array}
\right),
\qquad A_{\sigma_{2k}}=
\left(
\begin{array}{cc}
0 & 1 \\
1 & 0
\end{array}
\right),$$

$k=2,\ldots,n/2$.

\begin{figure}
\begin{center}
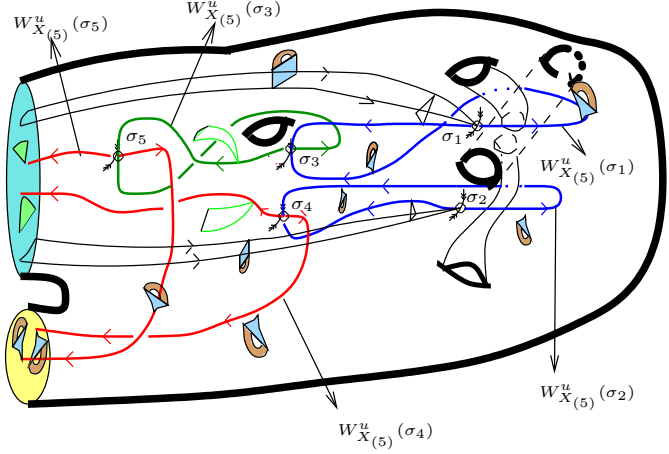
\caption{\label{Plug $X_5$} Plug $Z_5$.}
\end{center}
\end{figure}

We assume $X_{(n)}$ satisfying the hypotheses:

\textbf{($Z_n$1)}: There are two repelling periodic orbits $O_1$, $O_2$ in $Int(HB_{n+1})$ crossing the
holes of $R$.

\textbf{($Z_n$2)}: There are two solid tori neighborhoods $V_1, V_2 \subset Int(HB_{n+1})$ of $O_1,O_2$ with
boundaries transverse to $X_{(n)_t}$ such that if $N_n = HB_{n+1} \setminus (V_1 \cup V_2)$, then $N_n$ is a
compact neighborhood with smooth boundary transverse to $X_{(n)_t}$ and $X_{(n)_t}(N_n)\subset
N_n$ for $t>0$. $N_n$ is a handlebody of genus $n+1$ with two solid tori removed. 

\textbf{($Z_n$3)}: $R \subset N_n$ and the return map \index{Return map} $\tilde{G}$ induced by $X_{(n)}$ in $R$ satisfies 
the properties \textbf{(L1)-(L3)} given in \cite{ls}. Moreover,

$$
\{q\in N_n : X_{(n)_t}(q) \notin R, \forall t\in \re\} = Cl\left(\bigcup_{m=1}^{n-2} W^u_{X_{(n)}}(\sigma_m)\right).
$$

We define

$$ A_{X_{(n)}}^+=Cl\left(\bigcup_{t\in\re} X_{(n)_t}(A^+_{\tilde{G}})\right)\qquad
\text{and}\qquad A_{X_{(n)}}^-=Cl\left(\bigcup_{t\in\re} X_{(n)_t}(A^-_{\tilde{G}})\right).$$

$A_{X_{(n)}}^+$ and $A_{X_{(n)}}^-$ are homoclinic classes \index{Homoclinic classes} for $X_{(n)}$. Moreover
$A_{X_{(n)}}^+\cup A_{X_{(n)}}^-=
N_n(X_{(n)})$ and $A_{X_{(n)}}^+\cap A_{X_{(n)}}^-=O$, where $O=O_{X_{(n)}}(P)$ with $P$ the fix 
point associated to map $\tilde{G}$ defined in $R$.\\

The proof follows the same ideas to construct the example with four singularities.
\end{proof}




\section{Decomposing the maximal invariant set as union of $n$ homoclinic classes}
\label{eximpar}

From Theorem \ref{thpar} follows the first part of the main statement of this work. 
For these examples, the maximal invariant set can be
decomposed as union of two homoclinic classes. Now, will be proved for each $n\geq 2$, the
existence of a Venice mask such that the maximal invariant set is union of $n$ homoclinic classes.\\

As was observed in Section \ref{prel}, Venice masks containing one or three equilibria have already been developed. 
To continue, we provide an example with five singularities. The idea is very simple. We just proceed such as the process made
to obtain the vector field $X_{(3)}$.\\

First of all, the GLA as sectional-Anosov flow, is supported on a solid bitorus (see \cite{b}). The holes on the manifold
are produced because of the branches of the unstable manifold \index{Unstable manifold} of the saddle-type singularity.
Therefore, $X_{(3)}$ is a Venice mask defined on a handlebody of genus 4. The holes are generated by the branches of the
unstable manifolds of $\sigma_1$ and $\sigma_2$.\\

Now, for the vector field $X_{(3)}$, we add two Lorenz-like singularities \index{Lorenz-like singularity} located at the 
branches of $W^u_{X_{(3)}}(\sigma_2)$. We glue together in a $C^{\infty}$ fashion one copy of $GLA_{mod}$ along the
unstable manifold of the singularity $\sigma_2$. Thus is obtained the vector field $X_{(5)}$ whose flow is
depicted in Figure \ref{exa5}.\\

For each $i=1,2,3,$ there is a cross section \index{Cross section} $\Sigma_i$ and return map $F_i$ such that 

$$\Lambda_i=Cl\left(\bigcap_{n\geq 0} F_i^n(\Sigma_i)\right)$$

is a homoclinic class for $F_i$. Therefore

$$H_i=Cl\left(\bigcup_{t\in\re} X_{(5)}(\Lambda_i)\right)$$

is a homoclinic class for flow $X_{(5)}$. Moreover, $H_1\cap H_2\subset Cl(W^u_{X_{(5)}}(\sigma))$, $H_1\cap H_3\subset 
Cl(W^u_{X_{(5)}}(\sigma_2))$ and $H_2\cap H_3 \subset Cl(W^u_{X_{(5)}}(\sigma_2))$.

\begin{figure}
\begin{center}
{\scalebox{1.00}{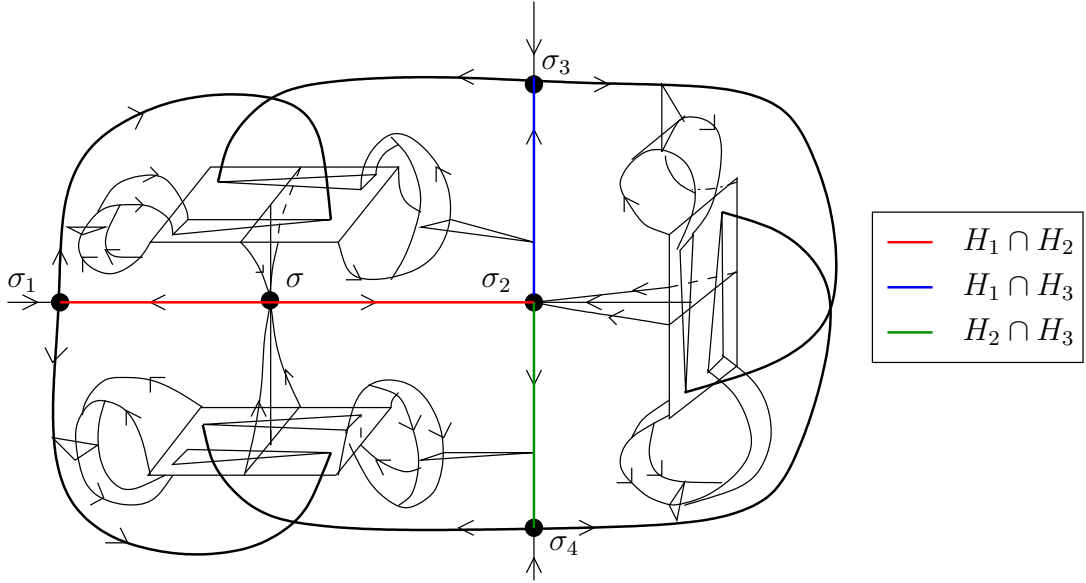}}
\caption{\label{exa5} Venice mask with five singularities.}
\end{center}
\end{figure}

\begin{prop}
$X_{(5)}$ is a Venice mask \index{Venice mask} supported on a handlebody $HB_6$ of genus 6. The maximal invariant set 
$HB_6(X_{(5)})$ is non-disjoint union of three homoclinic classes. \index{Homoclinic classes}
The intersection between two different homoclinic classes is contained
in $Cl(W^u(Sing(X_{(5)})))$.

\end{prop}

It is possible to continue gluing copies of $GLA_{mod}$ to produce Venice masks with any odd number of equilibria.
Each copy is glued along the unstable manifold of some singularity $\sigma_i$. The equilibrium $\sigma_i$ is chosen such that 
were previously possible to add two Lorenz-like singularities in its unstable manifold, one on each branch. More
specifically, each $\sigma_i$ is selected to add two new singular points if previously we have

$$A_{\sigma_{i}}\neq
\left(
\begin{array}{cc}
1 & 1 \\
1 & 1
\end{array}
\right).$$

In this way, the following theorem holds.

\pagebreak
\newpage

\begin{thm}
\label{thimpar}
For every $n$ odd ($n\geq 3$), there exists a Venice mask $X_{(n)}$ with $n$ singularities supported on a handlebody 
$HB_{n+1}$ of genus $n+1$. The maximal invariant set $HB_{n+1}(X_{(n)})$ is non-disjoint union of $(n+1)/2$ homoclinic
classes. The intersection between two different homoclinic classes is contained in $Cl(W^u(Sing(X_{(n)})))$.
\end{thm}

{\em Theorem} \ref{thF'} follows from {\em Theorem} \ref{thpar} and {\em Theorem} \ref{thimpar}.


\bibliographystyle{acm}
\bibliography{RCMBibTeX}

\begin{thebibliography}{10}

\bibitem{abs}
{\sc Afraimovich, V.~S., Bykov, V.~V., and Shilnikov, L.~P.}
\newblock On structurally unstable attracting limit sets of lorenz attractor
  type.
\newblock {\em Trudy Moskov. Mat. Obshch. 44}, 2 (1982), 150--212.

\bibitem{apu}
{\sc Arroyo, A., and Pujals, E.}
\newblock Dynamical properties of singular-hyperbolic at- tractors.
\newblock {\em Discrete Contin. Dyn. Syst. 19}, 1 (2007), 67--87.

\bibitem{b}
{\sc Bautista, S.}
\newblock The geometric lorenz attractor is a homoclinic class.
\newblock {\em Bol. Mat.(NS) 11}, 1 (2004), 69--78.

\bibitem{lec}
{\sc Bautista, S., and Morales, C.~A.}
\newblock Lectures on sectional-anosov flows.
\newblock http://preprint.impa.br/Shadows/SERIE{\_}D/2011/86.html.

\bibitem{bamo}
{\sc Bautista, S., and Morales, C.~A.}
\newblock On the intersection of sectional-hyperbolic sets.
\newblock {\em J. of Modern Dynamics 10}, 1 (2016), 1--16.

\bibitem{bmp}
{\sc Bautista, S., Morales, C.~A., and Pacifico, M.~J.}
\newblock On the intersection of homoclinic classes on singular-hyperbolic
  sets.
\newblock {\em Discrete and continuous Dynamical Systems 19}, 4 (2007),
  761--775.

\bibitem{bpv}
{\sc Bonatti, C., Pumari\~no, A., and Viana, M.}
\newblock Lorenz attractors with arbitrary expanding dimension.
\newblock {\em C. R. Acad. Sci. Paris S\'er. I Math. 325}, 8 (1997), 883--888.

\bibitem{gah}
{\sc G{\"a}hler, S.}
\newblock Lineare 2-normierte r{\"a}ume.
\newblock {\em Mathematische Nachrichten 28}, 1-2 (1964), 1--43.

\bibitem{gw}
{\sc Guckenheimer, J., and Williams, R.~F.}
\newblock Structural stability of lorenz attractors.
\newblock {\em Publications Math{\'e}matiques de l'IH{\'E}S 50}, 1 (1979),
  59--72.

\bibitem{hem}
{\sc Hempel, J.}
\newblock {\em 3-manifolds}.
\newblock Ann. of Math. Studies, No. 86. Princeton University Press, Princeton,
  N. J.; University of Tokyo Press, Tokyo, 1976.

\bibitem{hps}
{\sc Hirsch, M.~W., Pugh, C.~C., and Shub, M.}
\newblock {\em Invariant manifolds}, vol.~583.
\newblock Springer Berlin, 1977.

\bibitem{kat}
{\sc Kawaguchi, A., and Tandai, K.}
\newblock On areal spaces i.
\newblock {\em Tensor NS 1\/} (1950), 14--45.

\bibitem{ls}
{\sc L\'opez~Barragan, A.~M., and S\'anchez, H. M.~S.}
\newblock Sectional anosov flows: Existence of venice masks with two
  singularities.
\newblock {\em Bulletin of the Brazilian Mathematical Society, New Series 48},
  1 (2017), 1--18.

\bibitem{mem}
{\sc Metzger, R., and Morales, C.~A.}
\newblock Sectional-hyperbolic systems.
\newblock {\em Ergodic Theory and Dynamical Systems 28}, 05 (2008), 1587--1597.

\bibitem{mor3}
{\sc Morales, C.}
\newblock Singular-hyperbolic attractors with handlebody basins.
\newblock {\em Dyn. Control Syst. 13}, 1, 15--24.

\bibitem{exm}
{\sc Morales, C.~A.}
\newblock Examples of singular-hyperbolic attracting sets.
\newblock {\em Dynamical Systems 22}, 3 (2007), 339--349.

\bibitem{mo}
{\sc Morales, C.~A.}
\newblock Sectional-anosov flows.
\newblock {\em Monatshefte f{\"u}r Mathematik 159}, 3 (2010), 253--260.

\bibitem{mp2}
{\sc Morales, C.~A., and Pac{\'i}fico, M.~J.}
\newblock Sufficient conditions for robustness of attractors.
\newblock {\em Pacific journal of mathematics 216}, 2 (2004), 327--342.

\bibitem{mp}
{\sc Morales, C.~A., and Pac{\'i}fico, M.~J.}
\newblock A spectral decomposition for singular-hyperbolic sets.
\newblock {\em Pacific Journal of Mathematics 229}, 1 (2007), 223--232.

\bibitem{mpp4}
{\sc Morales, C.~A., Pac{\'i}fico, M.~J., and Pujals, E.~R.}
\newblock Robust transitive singular sets for 3-flows are partially hyperbolic
  attractors or repellers.
\newblock {\em Annals of mathematics\/} (2004), 375--432.

\bibitem{mv}
{\sc Morales, C.~A., and Vilches, M.}
\newblock On 2-riemannian manifolds.
\newblock {\em SUT J. Math. 46}, 1 (2010), 119--153.

\bibitem{r}
{\sc Reis, J.}
\newblock Infinidade de \'orbitas peri\'odicas para fluxos seccional anosov.
\newblock {\em Tese de Doutorado, UFRJ\/} (2011).

\bibitem{sma}
{\sc Smale, S.}
\newblock Differentiable dynamical systems.
\newblock {\em Bulletin of the American mathematical Society 73}, 6 (1967),
  747--817.

\end{thebibliography}


\medskip

\flushleft
S. Bautista.\\
Departamento de Matem\'aticas, Universidad Nacional de Colombia.\\
Bogot\'a, Colombia.\\
E-mail: sbautistad@unal.edu.co

\flushleft
A. M. L\'opez B.\\
Instituto de Ci\^encias Exatas \small{(ICE)}, Universidade Federal Rural do Rio de Janeiro.\\
Departamento de Matem\'atica, Serop\'edica, Brazil.\\
E-mail: barragan@im.ufrj.br

\flushleft
H. M. S\'anchez.\\
Departamento de Matem\'aticas, Universidad Central.\\
Bogot\'a, Colombia.\\
E-mail: hmsanchezs@unal.edu.co

\end{document}